\documentclass{amsart}
\RequirePackage{fix-cm}

\usepackage[utf8]{inputenc}
\usepackage{multirow}
\usepackage{nicefrac}
\usepackage{tikz-cd}
\usepackage[font={footnotesize}]{subcaption}
\numberwithin{equation}{section}
\usepackage{tikz}
\usetikzlibrary{positioning}
\usepackage{amssymb,amsmath,amsfonts,latexsym}
\usepackage{graphicx}
\usepackage[all]{xy}
\usepackage{kmath}
\usepackage{cite}
\usepackage[T1]{fontenc}
\usepackage{hyperref}
\usepackage[parfill]{parskip}
\usepackage{float}
\usepackage{xparse}
\usepackage{tikz}
\usepackage{newtxtext,newtxmath}
\usepackage{tkz-euclide}
\usetikzlibrary{decorations.markings,arrows}
\tikzset{
    arrowMe/.style={
        postaction=decorate,
        decoration={
            markings,
            mark=at position .5 with {\arrow[thick]{#1}}
        }
    }
}

\usetikzlibrary{backgrounds,calc}
\tikzset{point/.style = {fill=black,circle,inner sep=0.7pt}}
\pgfarrowsdeclaredouble{<<s}{>>s}{stealth}{stealth}%   double stealth
\pgfarrowsdeclaretriple{<<<s}{>>>s}{stealth}{stealth}% triple stealth
\tikzstyle{vertex}=[circle,fill=black!25,minimum size=20pt,inner sep=0pt]

\tikzstyle{edge} = [draw]
\tikzstyle{weight} = [font=\tiny]

\title{ From Dual Connections to Almost Contact Structures }

\author{Emmanuel Gnandi}
\address{ENAC, Universit\'e de Toulouse}
\email{kpanteemmanuel@gmail.com}

\author{Stéphane Puechmorel}
\address{ENAC, Universit\'e de Toulouse}
\email{stephane.puechmorel@enac.fr}

\theoremstyle{plain}
\newtheorem{thm}{Theorem}[section]
\newtheorem{cor}{Corollary}[section]
\newtheorem{lem}[thm]{Lemma}
\newtheorem{prop}{Proposition}[section]
\newtheorem{rem}[thm]{Remark}
\theoremstyle{definition}
\newtheorem{defn}[thm]{Definition}

\theoremstyle{remark}
\numberwithin{equation}{section}
\newcommand{\R}{\ensuremath{\mathbb{R}}}

\newcommand{\lc}{\ensuremath{\nabla^{lc}}}
\newcommand{\ricc}{\ensuremath{\text{Ric}}}
\newcommand{\met}[2]{\ensuremath{g\left(#1,#2\right)}}

\begin{document}
\maketitle
\textbf{Keywords:} dual connections, auto-dual connections, torsionless dual connections, auto-dual torsionless connection(Levi-Civita connection), gauge equation of dual connections, almost cosymplectic structure, almost symplectic structure, symplectic structure, cosymplectic structure, almost contact structure, almost contact metric structure, coKhaler structure.

\section{Introduction}
Finding characteristic obstructions to the existence of structures is a particularly important question 
arising in mathematics. In this work, we give conditions for an orientable manifold to admit an\textbf{ almost contact
structure(almost cosymplectic structure), almost contact metric structure, cosymplectic(symplectic mapping torus) structure}, using the notion of dual connections that was introduced in the context of information geometry
 \cite{amari1987differential,amari2012differential}. We  will also use information geometry to describe the relationships between the structures on an even dimensional manifold and the corresponding ones on an odd dimensional manifold.
Going back to the original paper \cite{sasaki1960}, 
 given  a differentiable manifold $M$ of odd dimension $2n+1$, 
an \textbf{almost contact structure} is defined by a triple $(\phi,\xi,\eta)$ with 
$\phi \in T_1^1(M), \xi \in T(M),\eta \in T^*(M)$ and such that:
\begin{align}
    & \eta(\xi) = 1 \label{subeq:eta_xi}\\
    & \phi^2 = -\text{Id} + \eta \otimes \xi \label{subeq:contact_def_prop}
\end{align}
A manifold with an almost contact structure can also be defined equivalently as one whose structure group is reductible
to $U(n) \times 1$. 

From the above equation, one can easily deduce the next proposition:
\begin{prop}
    \label{prop:contact_basic_prop}
    \begin{align}
  & \text{rank }\phi = 2n \label{subeq:contact_rank}\\
    & \phi \xi = 0 \label{subeq:phi_xi}\\
    & \eta \phi = 0  \label{subeq:eta_phi}
    \end{align}
\end{prop}
The proof is elementary and relies only on basic linear algebra.
In fact, if $p \in M$ and $X \in T_pM \neq 0 $ is such that $\eta(X) = 0$, then $\phi^2(X) = -X$, and $X \notin \ker \phi$.  
From \ref{subeq:contact_def_prop}, $\eta(\phi^2(X)) = -\eta(X) + \eta(\xi)\eta(X) = 0$ and using the previous remark it implies 
$\eta(\phi(X)) = 0$. Since $\phi^2(\xi) =0$, it comes at once that $\phi(\xi) = 0$.
A Riemannian metric $g$ on $M$ is said to be adapted to the almost contact structure if it satisfies for all vector fields $X,Y$:
\begin{equation}
    \label{eq:adapted_metric}
    g\left(\phi(X),\phi(Y)\right) = g(X,Y) - \eta(X)\eta(Y)
\end{equation}
Using \ref{subeq:phi_xi} and the above definition it comes:
\begin{equation}
    \label{eq:eta_xi_relationship}
    \eta(X) = g(X,\xi)
\end{equation}
and in turn:
\begin{equation}
    \label{eq:phi_xi_metric}
    g\left(\phi(X),\xi \right)=0
\end{equation}
The endomorphism $\phi$ is skew-symetric with respect to an adapted metric:
\begin{equation}
    \label{eq:skew_symetry}
    g\left(X,\phi(Y)\right)=-g\left(\phi(X),Y\right)
\end{equation}
and thus gives rise to a canonical $2$-form $\Omega$:
\begin{equation}
    \label{eq:2-form}
    \Omega(X,Y) = g\left(X, \phi(Y)\right)
\end{equation}
When $\Omega = d\eta$, the \textbf {almost contact structure} is said to be \textbf{a contact metric structure}. Finally, if $\xi$ is Killing, then 
the structure is said to be $K$-contact. Any $3$-dimensionnal oriented Riemannian manifold $(M,g)$ admits an \textbf{almost contact
 structure} with $g$ as adapted metric \cite{cabrerizo2009}. For classification of almost
contact metric structures (see also\cite{chinea1990classification}).

 In 1969 M. Gromov \cite{gromov1969stable} proved that, any \textbf{almost contact} open manifold 
M admits a contact structure. A similar result is proved in   case of closed oriented 3-dimensionnal manifold by Lutz \cite{lutz1970existence} and Martinet \cite{martinet1971formes}, the case of 5-dimensional is proved by J. Etnyre \cite{etnyre2012contact} and the work of
R. Casals, D.M. Pancholi and F. Presas \cite{casals2015almost}. In \cite{borman2015existence}, Matthew Strom Borman, Yakov Eliashberg and Emmy Murphy proved the same result in any dimension. 

 Being an almost contact manifold is a purely topological condition. In dimension 5, it boils down to the vanishing of the 
 third integral Stiefel-Whitney class. In \cite{Bar65}, this property is used to classify simply connected almost contact
 manifolds.
Recall that \textbf{almost cosymplectic manifold}(cf.  \cite{libermann1962quelques},\cite{albert1989theoreme}) of dimension $2n+1$ is a triple $(M, \omega,\eta)$ such that
the 2-form $\omega$ and the 1-form $\eta$ satisfy $\omega^n \wedge \eta \neq 0$.
In the language of $G$-structures, an \textbf{almost cosymplectic structure} can be defined equivalently as an $1 \times Sp(n,R)$-structure. 

 From \cite{albert1989theoreme}, every almost cosymplectic structure  on $M$ induces an isomorphism 
of $C^{\infty}(M)$-modules
\[
\flat_{(\eta,\omega)}:
\begin{cases}& \mathcal{X}(M)\rightarrow \Omega^{1}(M) \\
    & X \mapsto i_{X}\omega+\eta(X)\eta \\
    \end{cases}
    \] 
for every vector field $X\in \mathcal{X}(M)$. 
A vector bundle isomorphism (denoted with the same symbol) $\flat_{(\eta,\omega)}:TM\rightarrow T^{\star}M$ is also induced. 
Then the vector field  \[\xi=\flat^{-1}_{(\eta,\omega)}(\eta)\] on M is called the Reeb vector field of 
the almost cosymplectic manifold $(M, \eta, \omega)$ and is characterized by the following conditions
\[  i_{\xi}\omega=0 \quad\text{and}\quad i_{\xi}\eta=1           \]

Conversely, we have the following characterization of almost cosymplectic manifolds that follows from \cite{albert1989theoreme}[proposition 2]

\begin{prop}
Let M be a manifold endowed with a 1-form $\eta$ and a 2-form $\omega$ 
such that the map $\flat_{(\eta,\omega)}:TM\rightarrow T^{\star}M$ is an isomorphism.
Assume also that there exists a vector field $\xi$ such that $i_{\xi}\omega=0$ and $\eta(\xi)=1$. 
Then, M has odd dimension and $(M,\eta,\omega)$ is an almost cosymplectic manifold with Reeb vector field $\xi$.
\end{prop}
By a \textbf{cosymplectic manifold}, we mean a
(2n+1)-manifold M together with a closed 1-form $\eta$ and a closed 2-form $\omega$ such that
$\eta\wedge\omega^{n}$
 is a volume form. This was P. Libermann’s definition in 1959 \cite{libermann1959automorphismes}, under the
name of cosymplectic manifold. The pair $(\eta, \omega)$ is called a \textbf{cosymplectic structure}
on M. In \cite{blair2010riemannian}, Blair gives an equivalent definition of cosymplectic manifolds, which is more often referred to in the literature, see \cite{li2008topology},\cite{bazzoni2015k},\cite{chinea1993topology},\cite{cappelletti2013survey},\cite{fujimoto1974cosymplectic}, \cite{olszak1989locally}, \cite{olszak1981almost}. From Blair\cite{blair2010riemannian} an almost contact metric structure $(\theta,\xi,\eta,g)$ on an odd-dimensional smooth manifold M
is cosymplectic if $d\eta=d\Omega=0$, where $\Omega$ is the fundamental 2-form.\textbf{ The cosymplectic manifolds} can be thought of as an odd-dimensional counterpart of symplectic manifolds.
In fact, on any cosymplectic manifold $(M,\eta,\omega)$ the so-called horizontal distribution $\ker \eta$ is integrable to
a symplectic foliation of codimension 1. 
On the other hand, one has the following result due to Manuel de Léon and Martin Saralegi:
\begin{thm}[\cite{de1993cosymplectic}]
Let M be a manifold and $\omega$, $\eta$ two differential forms on M with degrees 2 and 
1 respectively. Consider, on $Y = M\times\mathbb{R}$, the differential 2-form $\Omega = pr^{\star}\omega + pr^{\star}\eta\wedge dt$, where $t\in\mathbb{R}$ and $pr:Y\rightarrow M.$ Then: $(M, \eta, \omega)$ is a cosymplectic manifold if and only if $(Y,\Omega)$ is a symplectic manifold.
\end{thm}
\textbf{The Darboux theorem} admits an equivalent in cosymplectic structure.\\Any cosymplectic manifold $(M, \eta, \omega)$ of dimension 2n + 1 admits around any point local coordinates
$(t, q^{\alpha}
, p_{\alpha}), \alpha = 1, . . . , n$, such that:
$$ \omega=\sum_{\alpha=1}^{n}dq^{\alpha}\wedge dp_{\alpha},\quad \eta=dt,\quad \xi=\frac{\partial}{\partial t}            $$
 In 2008, HONGJUN LI  main theorem in \cite{li2008topology} asserts that cosymplectic manifolds are equivalent to symplectic mapping tori. The main idea of Li's proof comes from the theorem of Tischler\cite{tischler1970fibering}, which states that: A compact manifold admits a non-vanishing closed
1-form if and only if the manifold fibres over a circle. This assertion is also equivalent to: A compact manifold is a mapping torus if and only if
it admits a non-vanishing closed 1-form. The codimension-one, co-orientable foliations defined by  the kernel of nowhere-zero closed  one form  are termed
unimodular foliations. In \cite{guillemin2011codimension} for the codimension one co-orientable, the existence of an unimodular foliation is equivalent to a vanishing modular class.
\begin{thm}[\cite{guillemin2011codimension}]
The first obstruction class(The modular class) $c_{\mathcal{F}}$ vanishes identically if and only if we can chose $\eta$ the defining one-form of the foliation $\mathcal{F}$ to be closed.
\end{thm}
In section 2, we briefly summarize results about the gauge equation for dual connections. 
There is no claim of originality here, only a reformulation of the previous results obtained by Pr. M. Boyom.
 In section 3, we discuss the relationship between skew-symmetric solutions of maximal rank of the gauge equation and
  the existence of almost cosympletic structure, almost contact metric structure and cosymplectic structure(symplectic mapping torus).
Finally, the case of dimension 3 coKähler manifolds is treated in the final part of the article.

\section{Gauge transformations and parallelism}
In this section, $(M,g)$ is a smooth Riemannian manifold. As usual, for a vector bundle $E \to M$,
 $\Gamma(E)$ denotes the space of smooth sections. 
 For any affine connection $\nabla$, its dual connection
 $\nabla^*$ is defined by the relation 
 $\left( \nabla_Y^*X \right)^\flat   = \nabla_YX^{\flat}$, or equivalently as satisfying 
 for any vector fields $X,Y,Z$ in $TM$, the equation:
 \begin{equation}
  \label{eq:dual_connection}
  Z\left( \met{X}{Y}\right) = \met{\nabla_Z X}{Y}  + \met{X}{\nabla_Z^* Y}
 \end{equation}
The equation \ref{eq:dual_connection} proves by symmetry that $\nabla^{**}=\nabla$.

On 1-forms, the duality relation  becomes
$\left( \nabla_X \omega \right)^\sharp = \nabla_X^* \omega^\sharp$, for any 1-form $\omega$ and vector
field $X$.

 The levi-civita connection $\lc$ is self-dual and for any connection $\nabla$ without torsion:
 \begin{equation}
  \label{eq:sum_dual}
  \nabla = \lc - \frac{1}{2} D, \, \nabla^* = \lc + \frac{1}{2} D
 \end{equation}
 with $D=\nabla^*-\nabla$ a $(2,1)$-tensor (since the difference of two affine connections is a tensor).
 
 The relationship between the curvatures of two dual connections is
given by :\[g(R^{\nabla}(X,Y)V,W)=-g(V,R^{\nabla^*}(X,Y)W).
\]
 A connection $\nabla$ in $TM$ is said to be metric if $\nabla g = 0$, i.e.:
 \[ X.(g(Y,Z))=g(\nabla_{X}Y,Z)+g(Y,\nabla_{X}Z),\quad \text{for any vector fields X,Y,Z}. \] 
 Metric connections are not unique, but differ only by the torsion. As a consequence of $\nabla g=0$ one has  $$ g(R^{\nabla}(X,Y)V,W)=-g(V,R^{\nabla}(X,Y)W). $$

\begin{prop}
  \label{prop:diff_nabla}
  $D$ is symmetric in its first two arguments. Furthermore, for any vector fields $X,Y,Z$:
  \[
  \met{D(Z,X)}{Y} = \met{X}{D(Z,Y)}
    \]
\end{prop}
  \begin{lem}
  \label{lem:dual_torsion}
  If $\nabla$ is torsionless, then so is $\nabla^*$.
 \end{lem}
 \begin{proof}
  Only a sketch of the proof is given here. The starting point is the same as for establishing Koszul formula:
  \begin{align*}
&   Y\left( \met{X}{Z} \right) = \met{\nabla_Y X}{Z} + \met{X}{\nabla_Y^*Z} \\
&   Z\left( \met{X}{Y} \right) = \met{\nabla_Z X}{Y} + \met{X}{\nabla_Z^*Y} \\
&   X\left( \met{Y}{Z} \right) = \met{\nabla_X Y}{Z} + \met{Y}{\nabla_X^*Z} \\
  \end{align*}
  It comes:
  \[
    \begin{split}
      & Y\left( \met{X}{Z} \right) - Z \left( \met{X}{Y} \right) + X\left( \met{Y}{Z} \right) = \\
      & -\met{\left( \nabla_X + \nabla_X^* \right)Z}{Y} + \met{[Y,X]}{Z} - \met{[Z,X]}{Y} \\
      & + \met{X}{\nabla_Y^* Z -\nabla_Z^*Y}
    \end{split}
    \]
    since $\nabla+\nabla^*=2 \lc$, Koszul formula yields:
    \[
      \met{X}{\nabla_Y^* Z -\nabla_Z^*Y} = \met{X}{[Y,Z]}
      \]
    and the lemma follows.
 \end{proof}
\begin{proof}
  The first claim is a consequence of $\nabla$ being torsionless and lemma: \ref{lem:dual_torsion}
  \[
    D(X,Y) = \nabla^*_X Y - \nabla_X Y = \nabla^*_Y X + [X,Y] - \nabla_Y X - [X,Y] = D(Y,X)
    \]
  For the second, the starting point is equation \ref{eq:dual_connection} rewritten with the expressions from equation \ref{eq:sum_dual}:
  \[
     Z\left( \met{X}{Y}\right) = \met{\lc_Z X}{Y}  + \met{X}{\lc_Z Y} - \frac{1}{2}\met{D(Z, X)}{Y} + \frac{1}{2}\met{X}{D(Z,Y)}
    \]
  Using the defining property of the Levi-Civita connection:
  \[
    \met{D(Z, X)}{Y} - \met{X}{D(Z,Y)} = 0
  \]
  and the claim follows.
\end{proof}
\begin{prop}
\label{prop:t_tensor}
The tensor:
\[
T \colon (X,Y,Z) \mapsto \met{D(Z,X)}{Y}
\]
is totally symmetric. Futhermore, $T(X,Y,Z) = (\nabla_Z g)\left(X,Y\right)
$
\end{prop}
\begin{proof}
  The symmetry comes from the one of $D$. For the second part of the proposition:
  \[
  \begin{split}
    (\nabla_Z g)\left(X,Y\right) & = Z\left(\met{X}{Y}\right) - \met{\nabla_Z X}{Y} - \met{X,\nabla_Z}{Y}  \\
    & = \met{\nabla_Z^*X}{Y} - \met{\nabla_Z X}{Y} \\
    & = \met{D(Z,X)}{Y}.
  \end{split}
  \]
\end{proof}
 
Given a torsionless connection $\nabla$, a $(1,1)$-tensor $\theta$ is said to satisfy the gauge equation if for all vector fields 
$X,Y$:
\begin{equation}
  \label{eq:gauge}
  \nabla^*_X \theta Y = \theta \nabla_X Y
\end{equation}
Equivalently, using the tensor $D$:
\begin{align}
  & \nabla \theta = - \left( D\otimes 1 \right) \theta  \label{eq:gauge_1_d}\\
  & \nabla^* \theta = - \left(  1\otimes  D\right)\theta \label{eq:gauge_d_1}\\ 
  & \left( \lc + \frac{1}{2} \left( 1 \otimes D + D \otimes 1 \right) \right) \theta = 0 \label{eq:gauge_parallel}
\end{align}
with:
\[
(D \otimes 1)(\theta)(X,Y) = D\left( X,\theta Y\right), \, (1\otimes D)(\theta)(X,Y) = \theta D(X,Y).\]

When $\nabla=\lc$, equation \ref{eq:gauge_parallel} yields: $\lc \theta = 0$. In this case,
 equation \ref{eq:gauge_parallel} indicates that local solutions exist provided the conditions of \cite{ATKINS2011310} are satisfied.
In coordinates, the gauge equation becomes, with
Einstein convention of summation on repeated indices:
\begin{equation}
  \label{eq:gauge_coord}
  \partial_k \theta_i^j = \Gamma_{ik}^b \theta_b^j - \Gamma_{ak}^j \theta_i^a - \theta_i^a D_{ak}^j
\end{equation}
where the $\Gamma_{ij}^k$ are the Christoffel symbols of $\nabla$. 
It is convenient to use an orthonormal frame $\left( X_1,\dots ,X_n \right)$
and its associated coframe $\left( \omega^1=X_1^\flat, \dots, \omega^n=X_n^\flat \right)$ to represent the tensor $D$:
\begin{equation}
  \label{eq:d_coefficients}
  D_{ij}^k =  \Gamma_{ij}^k + \Gamma_{ik}^j
\end{equation}
where all the coefficients are expressed in the orthonormal frame/coframe, that is:
\[
  D =  D_{ij}^k X_k\otimes\omega^i\otimes\omega^j
\]

\begin{defn}
  \label{def:tensor_adjoint}
  Let $\theta$ be a $(1,1)$-tensor. Its adjoint $\theta^*$ is defined, for all vector fields $X,Y$, 
  by the relation:
  \[
    \met{\theta X}{Y} = \met{X}{\theta^* Y}
    \]
\end{defn}
\begin{prop}
  \label{prop:gauge_ed_dual_theta}
  If $\theta$ is a solution of the gauge equation for $\nabla$, then so is its adjoint $\theta^*$.
\end{prop}
\begin{proof}
  For any vector fields $X,Y,Z$:
  \begin{align}
    \met{(\nabla_Z^* \theta)X}{Y} & = \met{\nabla_Z^*(\theta X)}{Y}-\met{\theta \nabla_Z^* X}{Y} \\
    & = Z\left( \met{\theta X}{Y} \right) - \met{\theta X}{\nabla_Z^Y}-\met{\theta \nabla_Z^* X}{Y} \\
    &=  Z\left( \met{X}{\theta^* Y} \right) - \met{X}{\theta^* \nabla_Z  Y}-\met{\nabla_Z^* X}{\theta^*Y} \\
    & = \met{X}{\nabla_Z\theta^*Y}-\met{X}{\theta^*\nabla_ZY}
  \end{align}
  Since $\theta$ satisfies the gauge equation, $\nabla_Z^* \theta = - \theta D(Z,.)$, thus:
\[
\met{(\nabla_Z^* \theta)X}{Y} = - \met{\theta D(Z,X)}{Y} = -\met{D(Z,X)}{\theta^*Y}=
-\met{X}{D\left( Z,\theta^*Y \right)}
\]
and so:
\begin{align}
  0 = &\met{X}{D\left( Z,\theta^*Y \right)} + \met{X}{\nabla_Z\theta ^*Y}-\met{X}{\theta^*\nabla_ZY}\\
  & = \met{X}{\nabla_Z^* \theta^*Y} - \met{X}{\nabla_Z \theta^*Y} + \met{X}{\nabla_Z\theta^*Y}-\met{X}{\theta^*\nabla_Z^*Y}\\
  & = \met{X}{\nabla_Z^* \theta^*Y}-\met{X}{\theta^*\nabla_Z Y}
\end{align}
This equation implies in turn the required property:
\[
  \nabla_Z^* \theta^* Y = \theta^* \nabla_Z Y
  \]
\end{proof}
\begin{rem}
  \label{rem:adjoint_petersen}
  This proposition generalizes theorem 10.3.2 in \cite{petersen2006riemannian}. It implies that if a tensor
  is a solution of the gauge equation, so are its symmetric and skew-symmetric parts. 
\end{rem}
\begin{prop}
  \label{prop:theta_parallel}
  Let $\theta$ be a skew-symmetric solution of the gauge equation. Let the tensor $p_\theta$ be defined for all vector fields $X,Y$ by:
  \[
  p_\theta(X,Y) = \met{\theta X}{Y}
\]
Then $p$ is $\nabla$ parallel, or equivalently, for any vector fields $X,Y,Z$:
\[
\left(\nabla_Z^* g\right)(\theta X, Y) = \met{\left(\nabla_Z \theta\right)X}{Y}
\]
\end{prop}
\begin{proof}
  For any vector fields $X,Y,Z$:
  \[
  \begin{split}
      \left(\nabla_Z p_\theta\right)(X,Y) & = Z\left(p_\theta(X,Y)\right) - p_\theta\left(\nabla_Z X,Y\right)
      - p_\theta\left(X,\nabla_Z Y\right) \\
      & = \met{\nabla_Z^* \theta X}{Y} + \met{\theta X}{\nabla_Z Y} - \met{\theta \nabla_Z X}{Y}
      - \met{\theta X}{\nabla_Z Y} \\
      &= \met{\left(\nabla^*_Z \theta - \theta \nabla_Z\right)X}{Y} = 0
  \end{split}
  \]
  On the other hand:
  \[
  \begin{split}
     \left(\nabla_Z^* g\right)(\theta X,Y)  & = Z\left(\met{\theta X}{Y}\right) - \met{\nabla_Z^* \theta X}{Y}
     - \met{\theta X}{\nabla_Z^* Y} \\
     & = \met{\nabla_Z \theta X}{Y} - \met{\nabla_Z^* \theta X}{\nabla_Z Y} \\
     & = -\met{D(Z,\theta X)}{Y}
  \end{split}
  \]
  and by the gauge equation:
  \[
  -\met{D(Z,\theta X)}{Y} = \met{(\nabla_Z \theta)X}{Y}
  \]
  proving the second assertion.
\end{proof}
\begin{cor}
  let $\theta$ be a solution of the gauge equation. Then the next two conditions are equivalent.
  \begin{enumerate}
      \item $\nabla\theta=0$
      \item $\nabla$ is metric connection, for the metric g.
  \end{enumerate}
\end{cor}
\begin{proof}
  By using the second assertion of Proposition 3.4, we have \[
\left(-\nabla_Z g\right)(\theta X, Y)=\left(\nabla_Z^* g\right)(\theta X, Y) = \met{\left(\nabla_Z \theta\right)X}{Y}
\]
The proposition is demonstrated.
\end{proof}  
\begin{rem}
  In the case of torsionless dual connections, $\nabla$ is exactly the Levi-Civita connection of the metric g. 
\end{rem}
\begin{cor}
   Let $\theta$ be a solution of the gauge equation of dual torsionless connections. The tensor $p_\theta$ is closed and $\nabla$-coclosed.
\end{cor}
\begin{proof}
  For a torsionless connection $\nabla$ and a $k$-form $\omega$:
  \[
    d\omega_{\theta}\left(X_0,\dots, X_k\right) = 
    \sum_{i=0}^k (-1)^i \left(\nabla_{X_i}\omega\right)\left(X_0,\dots,\hat{X}_i,\dots,X_k\right)
    \]
Since $\nabla p_{\theta} = 0$, the previous formula applied to $p_{\theta}$ shows that $dp_{\theta} = 0$. From \cite{opozda2015bochner}, the codifferential
relative to $\nabla$ acting on differential forms as follows:\[ \delta^{\nabla}\omega=-tr_{g}\nabla\omega    \]
the previous formula applied to $p_{\theta}$ shows that $\delta^{\nabla}p_{\theta}=0,$ then $p_{\theta}$ is $\nabla$-coclosed.
\end{proof}

\section{From Dual Connections to  Almost contact manifold  }
\subsection{\textbf{Gauge equation of dual connections .}}
\begin{thm}
\label{thm:cosymplectic_dual}
The following assertions are equivalent:
\begin{enumerate}
    \item $M$ of dimension $2n+1$ admits an almost cosymplectic structure(almost contact structure),
    \item The gauge equation of dual connections on $M$ admits a skew-symmetric solution $\theta$ such that $\text{rank }\theta= 2n.$
\end{enumerate}
\end{thm}

\begin{proof}
Let's prove the necessary part (1) implies (2): Assume that M admits an almost contact structure $(\omega,\eta)$,  there exists a vector field $\xi$ such that $i_{\xi}\omega=0\quad\text{and}\quad \eta(\xi)=1$. For all $x \in M$,  it exists an adapted frame $(X_{0},X_{1},..X_{n},\hat{X}_{1},.,\hat{X}_{n})$ of $T_{x}M$ such that $$X_{0}=\xi_{x}\quad\text{and}\quad (X_{1},..X_{n},\hat{X}_{1},..,\hat{X}_{n})\quad \text{is a symplectic basis of}\quad H=ker(\eta).
$$
The adapted coframe $\left( \alpha^0=X_0^\flat, \dots, \hat{\alpha}^{n}=\hat{X}_{n}^\flat \right)$ satisfy :
$$  \omega_{x}=\alpha^{1}\wedge\hat{\alpha^{1}}+.....+\alpha^{n}\wedge\hat{\alpha^{n} }\quad \text{and}\quad \eta_{x}=\alpha^{0}.$$ Let $(Y_{0},.,Y_{n},\hat{Y}_{1},...,\hat{Y}_{n})$ and $(X_{0},.,X_{n},\hat{X}_{1},...,\hat{X}_{n})$  be two adapted frames at x. we have 

\[   Y_{i}=C^{j}_{i}X_{j}   +D^{j}_{i}\hat{X}_{j}  \quad\text{and}\quad          \hat{Y}_{i}=-D^{j}_{i}X_{j}+C^{j}_{i}\hat{X}_{j}      \] 
where $C,D\in Gl(n,\mathbb{R}).$ Hence the two frames are related by the $(2n+1)\times(2n+1)$ matrix S:

\begin{center}
\[
\begin{pmatrix}C&D&0\\ 
-D&C&0\\ 
0&0&1\\ 
\end{pmatrix} 
\]
\end{center}
Since the structure group of $M$ is reducible to $\text{Sp}(n,\R) \times 1$, one can find a adapted connection $\nabla$ preserving $\omega,\xi$:
\[
\nabla\xi=0\quad\text{and}\quad \nabla\omega=0.
\] From \cite{blair2010riemannian}, to a almost cosymplectic structure $(\omega,\eta)$ there exists an almost contact metric structure $(\theta,\xi,\eta,g)$ on M with the same $\xi$ {and} $\eta$, whose fundamental 2-form $\Omega$ coincides with $\omega.$ we define a metric
g on M by $$g(X,Y)=g_{H}(X,Y),\quad g(X,\xi)=0,\quad g(\xi,\xi)=1,\quad \forall X,Y\in \Gamma(H).$$

 The (1, 1)-tensor $\theta:TM\longrightarrow TM$ is defined by:
$$\theta X=JX,\quad \theta\xi=0  \quad \forall X\in\Gamma(H) $$
where $J^{2}X=-Id_{H}$ where $Id_{H}$ denotes the identity map on $H$ and $g_{H}$ is a metric on H such that $$\Omega(X,Y)=g_{H}(JX,Y)\quad \forall X,Y\in \Gamma(H).$$  \\
We have $$ \omega(X,Y)=\Omega(X,Y)=g(\theta X,Y) \quad\text{and}\quad g(\theta X,Y)=-g(X,\theta Y)  $$

We have 
\[\nabla\omega=\nabla\Omega=0
\]
\[X.\Omega(Y,Z)-\Omega(\nabla_{X}Y,Z)-\Omega(Y,\nabla_{X}Z)=0
\]

\[X.g(\theta Y,Z)-g(\theta\nabla_{X}Y,Z)-g(\theta Y,\nabla_{X}Z)=0
\]
By duality between $\nabla,\nabla^*$, we have 
\[
g(\nabla^{*}_{X}\theta Y,Z)-g(\theta\nabla_{X}Y,Z)=0
\]
we deduce that \[  \nabla^{*}_{X}\theta Y=\theta\nabla_{X}Y \quad\text{and}\quad g(\theta X,Y)=-g(X,\theta Y)  \]

So $\theta$ is skew-symmetric solution of the gauge equation such that $rank(\theta)=2n.$\\
The sufficient part (2) implies (1):\\   Let $\theta$ be a skew-symmetric solution of the gauge equation. By assumption the rank of $\theta$ is 2n, so 
2-form $p_{\theta}$ has maximal rank, i.e. $p_{\theta}^{n}$  vanishes
nowhere. Associated to $p_{\theta}$ is its 1-dimensional kernel distribution $ker p_{\theta}$. Since $M$ is orientable, by using the Hodge operator $\star$ on $M$, we define a one form $\eta_{\theta}$ such that :
$\eta_{\theta}=^{\star}p_{\theta}^{n}\quad\text{and satisfy naturaly}\quad p_{\theta}^{n}\wedge\eta_{\theta}\ne0.$ The 2-form 
$p_\theta$ defines a line bundle $l_{p_\theta}= \cup_{p \in M} \{p, \ker p_\theta\}$.
Let $\xi_\theta$  be  the unique  section  of $l_{p_{\theta}}$ satisfying $ i_{\xi_\theta}\eta_{\theta}=1.$ The one-form $\eta_{\theta}$ induces an hyperplane distribution by:$H^{ \eta_{\theta}}=ker \eta_{\theta} $
which is everywhere transverse to $l_{p_{\theta}}$. We see that $(p_{\theta},\eta_{\theta})$ determines a splitting \[   TM=( l_{p_{\theta}},\xi_{\theta})\oplus ({H}^{ \eta_{\theta}},\hat{p}_{\theta})           \]

 of the tangent space of M into a framed line bundle and a almost-symplectic hyperplane-bundle $(
 {H}^{ \lambda_{\theta}},\hat{p}_{\theta})$, where $\hat{p}_{\theta}$ is the restriction of $p_{\theta}$ to ${H}^{ \eta_{\theta}}$.\\
 \end{proof}
 
 \begin{cor}
 \label{cor:cosymplectic_selfdual}
 In almost cosymplectic manifold $(M,\omega,\eta)$, with $M$ of dimension $2n+1$, there are always dual connections $(\nabla,\nabla^{*})$ adapted to the distributions $ker\omega$ and $ker\eta$, that is $\nabla\Gamma^{\infty}(ker\omega)\subset \Gamma^{\infty}(ker\omega)\quad \text{and}\quad\nabla^{*}\Gamma^{\infty}(ker\eta)\subset \Gamma^{\infty}(ker\eta).$
 \end{cor}
\begin{proof}
Let $(\omega,\eta)$ be an almost cosymplectic structure on $M$. It exists $\nabla$ such that:
\[
\nabla\omega=0,\quad \nabla\xi=0.
\]
(i)$\nabla\omega=0$, let $Y\in \Gamma^{\infty}(ker\omega)$. By using  the identity 
\[
X.\omega(Y,Z)-\omega(\nabla_{X}Y,Z)-\omega(Y,\nabla_{X}Z)=0.
\]
we have 
\[
\nabla\Gamma^{\infty}(ker\omega)\subset \Gamma^{\infty}(ker\omega).
\]
(ii)$\nabla\xi=0$, By duality, we have:  
\[
(\nabla^{*}\eta)(X,Y)=g(\nabla_{X}\xi,Y).
\]
So $\nabla^*\eta=0$. By a simple calculation we have:
\[
\nabla^{*}\Gamma^{\infty}(ker\eta)\subset \Gamma^{\infty}(ker\eta).
\]
\end{proof}

\begin{cor}
Let $M$ be a manifold of dimension $2n+1$, Let put $W=M\times \mathbb{R}$. The following assertions are equivalent:
\begin{enumerate}
    \item The gauge equation of dual connections on $M$ admits a skew-symmetric solution $\theta$ such that $\text{rank }\theta= 2n,$
    \item $M$ admits an almost cosymplectic structure(almost contact structure),
    \item $W$ admits an almost symplectic structure,
    \item  The gauge equation of dual connections on $W$ admits a skew-symmetric solution $\theta$ such that $\text{rank }\theta= 2n+2.$
    
\end{enumerate}

\end{cor}

\begin{proof}
(1)$\iff$(2) is exactly the assertion of the previous theorem. \\
Let us proves that (2)$\iff$(3):\\
The necessary part “$(2)\Longrightarrow (3)$”:Starting from a almost cosymplectic structure $(\omega,\eta)$, from \cite{blair2010riemannian} there exist an almost contact metric  $(\theta,\xi,\eta,g)$ on M associated to the almost cosymplectic structure , from \cite{sasaki1961differentiable}
we known that  $W=M\times\mathbb{R}$ admits a almost complex structure J defined by:\[J(X,f\frac{\partial}{\partial s})=(\theta X-f\xi,\eta(X)\frac{\partial}{\partial s}).\] we know that from \cite{ehresmann1950varietes} the existence on a manifold
of almost complex structures is equivalent to  almost symplectic structures.\\
The sufficient part “(2)$\Longleftarrow$(3) ”Let  denote by $p:W=\mathbb{R}\times M\rightarrow M$ the
canonical projection and by $l(a) = (0, a):M\rightarrow W=\mathbb{R}\times M$
 a fixed section. Let $\Omega$ is  almost symplectic 2-form on $W$ ie ($\Omega^{n+1}\ne 0)$, let $s$ the coordinate in $\mathbb{R}$ and $\frac{\partial}{\partial s}$ the corresponding coordinate vector field on $\mathbb{R}$, we define $(\eta,\omega)$ by $$ \omega=l^{*}\Omega\quad, \eta=l^{*}i_{\frac{\partial}{\partial s}}\Omega.$$ We claim that on $W=M\times \mathbb{R}$ we have \[ \Omega=p^{*}\omega+p^{*}\eta\wedge ds.     \]
 Then from \cite{de1993cosymplectic}, we known that $\Omega^{n+1}=(n+1)p^{*}(\eta\wedge\omega^{n})\wedge ds.$ The 2-form $\Omega$ satisfies $\Omega^{n+1}\ne0$ thus $\eta\wedge\omega^{n}$ is volume form on $M$, and consequently the pair $(\eta,\omega)$ is almost cosymplectic structure on M.
 
 Let us proves that $(3)\iff(4):$
 The necessary part “$(3)\Longrightarrow (4)$”. Let $\Omega$ an almost symplectic on $W$, from \cite{vaisman1985symplectic},\cite{bourgeois1999variational}, there exist almost-symplectic connections $\nabla$ defined by \[
\nabla_{X}Y=\nabla^{0}_{X}Y+A(X,Y)
\]
Where $\nabla^{0}$ is the linear connection on $W$, defined by: \[\nabla^{0}_{X}\Omega(Y,Z)=\Omega(A(X,Y),Z). \] The almost-symplectic connections satisfy: \[ \nabla\Omega=0       \]
There exist a skew symmetric $\theta\in \Gamma(TW^{\star}\otimes TW)$  and Riemannian on $W$ such that the identity:
\[ \Omega(X,Y)=g(\theta X,Y) ,\quad \theta^{2}=-Id_{TW}  
\].

The identity $$ \nabla\Omega=0  $$ implies that :
\[X.\Omega(Y,Z)-\Omega(\nabla_{X}Y,Z)-\Omega(Y,\nabla_{X}Z)=0
\]
\[
X.g(\theta Y,Z)-g(\theta\nabla_{X}Y,Z)-g(\theta Y,\nabla_{X}Z)=0
\]
\[
g(\nabla^{*}_{X}\theta Y-\theta\nabla_{X}Y,Z)=0
\]
so we have \[\nabla^{*}_{X}\theta Y=\theta\nabla_{X}Y\quad\text{and}\quad rank(\theta)=rank(\Omega)=2n+2.
\] 
The sufficient part “(3)$\Longleftarrow$(4)\\
Let $\theta$ be a skew-symmetric solution of the gauge equation  of  dual connections $({\nabla},{\nabla}^{*})$ on $W=M\times \mathbb{R}$ of  $rank(\theta)=2n+2.$ The 2-form $p_{\theta}$ is non-degenerate on $W$, then $p_{\theta}$ is almost symplectic structure on $W$.
\end{proof}
Proceeding the same way, we have the following corollary:
\begin{cor}
Let $M$ be an even dimensional manifold of dimension $2n$, Let put $W=M\times \mathbb{R}$. The following assertions are equivalent:
\begin{enumerate}
    \item The gauge equation of dual connections on $M$ admits a skew-symmetric solution $\theta$ such that $\text{rank }\theta= 2n,$
    \item $M$ admits an almost symplectic structure(almost contact structure),
    \item $W$ admits an almost cosymplectic structure(almost contact structure),
    \item  The gauge equation of dual connections on $W$ admits a skew-symmetric solution $\theta$ such that $\text{rank }\theta= 2n.$
    
\end{enumerate}
\end{cor}

\begin{prop}
 Let $(\theta,\eta,\xi)$ be an almost contact manifold. The following assertions are equivalent:
\begin{enumerate}
    \item $\nabla\theta=0,\quad \nabla\xi=0,$
    \item $\nabla\theta=0,\quad \nabla\eta=0.$
\end{enumerate}
\end{prop}
\begin{proof}
Let $(\theta,\eta,\xi)$ an almost contact structure ie :
\[  \theta\circ\theta+I=\eta\otimes\xi,\quad \lambda(\xi)=1               \]
By a simple calculations, we have
$$(\nabla_{X}\theta)(\theta Y)+\theta((\nabla_{X}\theta)Y) =\eta(Y)(\nabla_{X}\xi)+((\nabla_{X}\eta)Y)\xi.                     $$ we deduce the equivalence.
\end{proof}
\begin{prop}
Let $(\omega,\eta)$ be an almost cosymplectic manifold with associated
almost contact metric structure $(\theta,\eta,\xi,g)$. If $\nabla\omega=0$, then the next assertions are equivalent:
\begin{enumerate}
    \item $\nabla\theta=0,$
    \item g is $\nabla$-paralell ie $(\nabla g=0)$,
    \item $\left( \nabla_X\xi \right)^\flat   = \nabla_X\eta$\quad\text{or}\quad $\nabla_{X}\xi=(\nabla_{X}\eta)^{\sharp}.$
\end{enumerate}
\end{prop}
\begin{proof}
Let proves that (1)$\iff$(2):\\
Let us proves that (2)$\Longrightarrow$(1):For any $X,Y,Z$, it comes:
  \[
    \begin{split}
    & \nabla_Z(\omega)(X,Y) =Z(g(\theta X,Y))-g(\theta\nabla_{Z}X,Y)-g(\theta X,\nabla_{Z}Y)\\
    & =g(\nabla_{Z}\theta X,Y)+g(\theta X,\nabla_{Z}Y)-g(\theta\nabla_{Z}X,Y)-g(\theta X,\nabla_{Z}Y)\\
    & =g(\nabla_{Z}\theta X,Y)-g(\theta\nabla_{Z}X,Y)\\
     & =g((\nabla_{Z}\theta)X,Y)\\
\end{split}
\] So we deduce the necessary part.\\ Let proves the sufficient part (1)$\Longrightarrow$(2)
Recall that from Proposition 3.4, $\nabla\omega=0$ is equivalent to $ \nabla^{*}_{X}\theta Y=\theta\nabla_{X}Y$, (1) implies $\nabla_{X}\theta Y=\nabla^{*}_{X}\theta Y$, we deduce that $ \nabla=\nabla^{*} $, then $\nabla g=0.$ This proves the sufficient part.\\
Let proves that (1)$\iff$(3):\\ Let proves the sufficient part (1)$\Longrightarrow$(3): Assume that $\nabla\theta=0$, then $\nabla=\nabla^{*}$, by using the formula $\left( \nabla_Y^*X \right)^\flat  = \nabla_YX^{\flat}$(resp,$\left( \nabla_X \omega \right)^\sharp = \nabla_X^* \omega^\sharp$), we deduce that $\left( \nabla_X\xi \right)^\flat   = \nabla_X\eta$(resp, $\nabla_{X}\xi=(\nabla_{X}\eta)^{\sharp}$).\\ Let proves the necessary part (3)$\Longrightarrow$(1):By simple observations  $\left( \nabla_X\xi \right)^\flat = \nabla_X\eta=\left( \nabla^{*}_X\xi \right)^\flat$, so $\nabla=\nabla^{*}$, then (1) is demonstrated.
\end{proof}

\subsection{\textbf{Gauge equation of selfdual connections}}
When $\nabla=\nabla^{*}$, the gauge equation is equivalent to \[(\nabla_X\theta)Y=0 \quad \forall X,Y\in \mathcal{X}(M).\]

\begin{thm}
The following assertions are equivalent:
\begin{enumerate}
    \item $M$ admits an almost contact metric structure,
    \item It exists a metric on $M$ such that the gauge equation of self dual connections with respect to it admits a skew-symmetric solution $\theta$ such that $\text{rank }\theta= 2n.$
\end{enumerate}
\end{thm}

\begin{proof}
This is essentially a corollary of theorem \ref{thm:cosymplectic_dual}.
Let proves that (1) implies (2).\\ Let $(\theta,\xi,\eta,g)$-structure(almost contact metric structure) on M, from \cite{sasaki1961differentiable}[Theorem 11],\cite{motomiya1968study}[Theorem 2], there exist an linear connection such that:\[\nabla\xi=0\quad,\nabla\theta=0,\quad \nabla\eta=0,\quad \nabla g=0 . \]
We deduce that $\theta$ is skew-symmetric solution of the selfdual connection $\nabla$ and the rank($\theta$)=2n.\\
Let proves that (2) implies (1).\\ Let $\theta$ be a skew-symmetric solution of the gauge equation  of  selfdual connections $\nabla$ such that $\text{rank }\theta= 2n.$
From \ref{thm:cosymplectic_dual}, M admits an almost cosymplectic structure. From \cite{blair2010riemannian}, there exists an almost contact metric structure $(\theta,\xi,\eta, g)$ on M.

\end{proof}

\begin{cor}
Let $M$ be a $2n+1$ dimensional manifold, Let put $W=M\times \mathbb{R}$, the following assertions are equivalents:
\begin{enumerate}
    \item The gauge equation of selfdual connections on $M$ admits a skew-symmetric solution $\theta$ such that $\text{rank }\theta= 2n,$
    \item $M$ admits an almost contact metric structure,
    \item $ W=M \times \mathbb{R}$ has an almost Hermitian structure, 
    \item The gauge equation of selfdual connections on $W$ admits a skew-symmetric solution $\theta$ such that $\text{rank }\theta= 2n+2.$
\end{enumerate}
\end{cor}

\begin{proof}
(1)$\iff$(2) is exactly the assertion of the previous theorem. Let us proves that $(2)\iff(3):$

 The necessary part “$(2)\Longrightarrow (3)$” Let $(\theta,\xi,\eta,g)$ be a almost contact metric structure on $M$, from \cite{blair2010riemannian} the pair $(J,h)$ where $J$ is  almost complex structure  defined by:$J(X,f\frac{\partial}{\partial s})=(\theta X-f\xi,\eta(X)\frac{\partial}{\partial s})$ and $h=g+dt^{2}$ is a product metric on $W$, we have $h(J(X,f\frac{\partial}{\partial t}),J(Y,f\frac{\partial}{\partial t}))=h((X,f\frac{\partial}{\partial t}),(Y,f\frac{\partial}{\partial t}))$, the pair $(J,h)$ is an almost Hermitian structure in $W$. 
 
 The sufficient part “(2)$\Longleftarrow$(3)"Let $(J,h)$ an almost Hermitian structure on $W$. The almost Hermitian form  defined by $\Omega(X,Y)=h(JX,Y)$ is a non-degenerate 2-form on $W$. Let $s$ the coordinate in $\mathbb{R}$ and $\frac{\partial}{\partial s}$ its coordinate vector field on $\mathbb{R}$. We define $(\eta,\omega)$ by:
 $$ \omega=l^{*}\Omega\quad, \eta=l^{*}i_{\frac{\partial}{\partial s}}\Omega.$$ 
 where the
canonical projection and by $l(a) = (0, a):M\rightarrow W=\mathbb{R}\times M$. The pair $(\omega,\eta)$ is an almost cosymplectic structure on$M$.
From \cite{blair2010riemannian} there exists an almost contact metric structure $(\theta,\xi,\eta, g)$ on $M$.

(3)$\iff$(4) :The necessary part “$(3)\Longrightarrow (4)$” Let $(J,h)$ be an almost Hermitian structure on $W$. From \cite{obata1956affine}[Theorem 15.1, corolarry 1] almost Hermitian connections exist, namely, linear connections $\nabla$ ( Bismut connection, Chern connection) defined by:\[     \nabla=\nabla^{h}-\frac{1}{2}J\nabla^{h}J               \]
satisfying:
\[  \nabla J=0 \quad\text{and}\quad \nabla h=0 .\]  
Then the gauge equation of selfdual connections on $M$ admits a skew-symmetric solution $J$ such that $\text{rank }J= 2n+2.$ 

The sufficient part “(3)$\Longleftarrow$(4)"Let $\theta$ be a skew-symmetric solution of the gauge equation  of  selfdual connections $\nabla$ on $W=M\times \mathbb{R}$ of the $rank(\theta)=2n+2.$ The 2-form $p_{\theta}$ is non-degenerate on $W$. There exist on $W$ an almost Hermitian structure $(J,h)$ such that $p_{\theta}(X,Y)=h(JX,Y).$ 
\end{proof}

\subsection{\textbf{Gauge equation of torsionless dual connections , modular class and cosymplectic manifold(symplectic mapping torus)}}

\begin{thm}
\label{thm:gauge_dual_torsionless}The following assertions are equivalent:
\begin{enumerate}
    \item $M$ admits an cosymplectic structure(symplectic mapping torus),
    \item  The gauge equation of dual torsionless connections admits a skew-symmetric solution $\theta$ such that $\text{rank }\theta= 2n$ and the modular class of the image of $\theta$ vanishes.
\end{enumerate}
\end{thm}
\begin{proof}
Let us proves that (1) implies (2)\\ 
Assume that M admits a cosymplectic structure $(\omega,\eta)$, with
$d\eta=0, d\omega=0$, such that $\eta\wedge\omega^{n}\ne 0$ is a volume-form. From \cite{blair2010riemannian}, it exists an almost
contact metric structure $(\theta,\xi,\eta,g)$ on $M$, where $\xi$ is the Reeb vector field defined by $i_{\xi}\omega =0\quad\text{and}\quad\eta(\xi) = 1)$ and $(\theta, g)$ may be obtained by polarizing $\omega$ on
the codimension one foliation $\text{H}=ker(\eta)$. It satisfies the following identities:
\[
\eta(\xi) = 1, \theta^2 = -\text{Id} + \eta \otimes \xi, g(\theta X,\theta Y)=g(X,Y)-\eta(X)\eta(Y) \]

The fundamental 2-form $\Omega$ of the almost contact metric structure  coincides with $\omega$, so we have \[\Omega(X,Y)=\omega(X,Y)=g(\theta X,Y).
\]
The condition $\eta\wedge\omega^{n}\ne 0$ implies that the restriction of $\omega$ to the leaves of the codimension one foliation $\text{H}=ker(\eta)$ is symplectic form. From \cite{bieliavsky2006symplectic} the connections $\nabla$ define by \[
\nabla_{X}Y=\nabla^{0}_{X}Y+\frac{1}{3}N(X,Y)+\frac{1}{3}N(Y,X)\quad \forall X,Y\in \Gamma(\text{H}).
\] is symplectic connections on H,
where $\nabla^{0}$ is any torsionless linear connection H, define by: \[\nabla^{0}_{X}\omega(Y,Z)=\omega(N(X,Y),Z) \quad \forall X,Y,Z\in \Gamma(\text{H})   \]
According to the decomposition of the tangent bundle as: $$TM=C^{\infty}(M)\xi\oplus\text{H}$$
where $\pi:TM\longrightarrow \text{H}$ denote the corresponding projection.
The symplectic connections $\nabla$ admits a torsionless lift $\Tilde{\nabla}$:
\[ \nabla:=\pi\Tilde{\nabla}|_{H}\quad\text{and}\quad \Tilde{\nabla}\xi=0.         \] 
Please note that $\nabla\omega=0$ implies  :
\[
\Tilde{\nabla}\omega=0\quad\text{and}\quad\Tilde{\nabla}\xi=0.
\]

We have, by using Blair's definition:
$$ \omega(X,Y)=\Omega(X,Y)=g(\theta X,Y) \quad\text{and}\quad g(\theta X,Y)=-g(X,\theta Y)  $$
It comes:
\begin{align*}
&\Tilde{\nabla}\omega=\Tilde{\nabla}\Omega=0\\
&X.\Omega(Y,Z)-\Omega(\Tilde{\nabla}_{X}Y,Z)-\Omega(Y,\Tilde{\nabla}_{X}Z)=0\\
&X.g(\theta Y,Z)-g(\theta\Tilde{\nabla}_{X}Y,Z)-g(\theta Y,\Tilde{\nabla}_{X}Z)=0\\
&g(\Tilde{\nabla}^{*}_{X}\theta Y,Z)-g(\theta\Tilde{\nabla}_{X}Y,Z)=0
\end{align*}
we deduce that \[  \Tilde{\nabla}^{*}_{X}\theta Y=\theta\Tilde{\nabla}_{X}Y \quad\text{and}\quad g(\theta X,Y)=-g(X,\theta Y)  \]

So $\theta$ is skew-symmetric solution of the gauge equation of torsionless dual connections $(\Tilde{\nabla},\Tilde{\nabla}^{*})$ such that $rank(\theta)=2n.$\\ By a simple observations, we have: \[ ker(\theta)=ker(\omega).\] We  deduce that \[im(\theta)=ker(\omega)^{\perp}=ker(\eta).\]
Then from \cite{guillemin2011codimension}, the modular class of the image of $\theta$ vanishes.\\
Let now proves that (2) implies (1).\\
 Let $\theta$ be a skew-symmetric solution of the gauge equation  of torsionless dual connections $({\nabla},{\nabla}^{*})$. From corollary 2.2, $p_{\theta}$ is $\nabla$-parallel, therefore it is closed. By assumption the rank of $\theta$ is 2n, so 
2-form $p_{\theta}$ has maximal rank, i.e. such that $p_{\theta}^{n}$  vanishes
nowhere. We associate to $p_{\theta}$ a one-dimensional foliation $ker p_{\theta}=ker(\theta).$ From \ref{prop:theta_parallel}
 $p_{\theta}$ is $\nabla$-parallel, so the foliation $ker p_{\theta}$ is $\nabla$-parallel i.e. ($\nabla\Gamma(kerp_{\theta})\subset\Gamma(kerp_{\theta})).$ By using the duality of  $({\nabla},{\nabla}^{*})$:\[
 X.g(v,v^{\perp})=g(\nabla_{X}v,v^{\perp})+g(v,\nabla^{*}_{X}v^{\perp})
 \]
we deduce that $\text{im}(\theta)$ is $\nabla^{*}$-parallel i.e. $\nabla^{*}\Gamma(im{\theta})\subset\Gamma(im{\theta})$.
By using the orientation on $M$ together with $p_{\theta}^{n}$, we orient $ker p_{\theta}.$ So $\text{im}(\theta)$ is transversally codimension one foliation. By assumption  the modular class of the image of $\theta$ vanishes, from \cite{guillemin2011codimension} there exist a closed one form $\eta_{\theta}$ on M suchr that im($\theta$)=$ker\eta_{\theta}$. We deduce that $(p_{\theta},\eta_{\theta})$ is cosymplectic structure on M.
\end{proof}

Proceeding the same way as corollary \ref{cor:cosymplectic_selfdual}, we have  
 \begin{cor}
 In cosymplectic manifold $(M^{2n+1},\omega,\eta)$, there are always dual connections torsionless $(\nabla,\nabla^{*})$ adapted to the distributions $ker(\omega)$ and $ker(\eta).$ By adapted we means that $$\nabla(\Gamma^{\infty}ker(\omega))\subset \Gamma^{\infty}ker(\omega) \quad, \nabla^{*}(\Gamma^{\infty}ker(\eta))\subset \Gamma^{\infty}ker(\eta).$$
 \end{cor}

Using the same technique as in the proof of theorem \ref{thm:gauge_dual_torsionless}, it comes:
\begin{cor}
Let $M^{2n+1}$ be an odd dimensional manifold, Let put $W=M^{2n+1}\times \mathbb{S}^{1}$, the following assertions are equivalents:
\begin{enumerate}
    \item The gauge equation of dual torsionless connections on $M^{2n+1}$ admits a skew-symmetric solution $\theta$ such that $\text{rank }\theta= 2n$ and the modular class of image of $\theta$ vanishes,
    \item $M^{2n+1}$ admits a cosymplectic structure,
    \item $W$ admits an  symplectic structure,
    \item  The gauge equation of dual torsionless connections on $W$ admits a skew-symmetric solution $\theta$ such that $\text{rank }\theta= 2n+2.$
    
\end{enumerate}

\end{cor}

\begin{cor}
Let $M^{2n}$ be an even-dimensional manifold, let put $W=M^{2n}\times \mathbb{S}^{1}$, the following assertions are equivalents: 
\begin{enumerate}
    \item The gauge equation of dual torsion-less connections on $M^{2n}$ admits a skew-symmetric solution $\theta$ such that $\text{rank }\theta= 2n.$ 
    \item $M^{2n}$ admits a symplectic structure.
    \item $W$ admits an  cosymplectic structure
    \item  The gauge equation of dual torsionless connections on $W$ admits a skew-symmetric solution $\theta$ such that $\text{rank }\theta=2n$ and the modular class of image of $\theta$ vanishes.
    
\end{enumerate}
\end{cor}

\begin{proof}
(1)$\iff$(2) is exactly the same intuition as (3)$\iff$(4) of the previous corollary (4.4). (3)$\iff$(4) is the same intuition as (1)$\iff$(2) of the previous corollary (4.4). Let us proves that (2)$\iff$(3):“\\ The necessary part $(2)\Longrightarrow (3)$”
Let $(M^{2n},\Omega)$ be a symplectic manifold, consider the symplectic mapping torus $W=M^{2n}_{\varphi}=\frac{M^{2n}\times [0,1]}{(m,0)\sim (\varphi m,1)}$, where $\varphi$ is a symplectic diffeomorphism. From \cite{li2008topology}  $W$ admits an cosymplectic structure. Let take $\varphi=Id$ then $W=M^{2n}\times \mathbb{S}^{1}=\frac{M^{2n}\times [0,1]}{(m,0)\sim (m,1)}$ admits a cosymplectic structure.\\ 
The sufficient part “(2)$\Longleftarrow$(3)\\ Let $(\Omega,\eta)$ a cosymplectic on $W=M^{2n}\times \mathbb{S}^{1}$, let consider fibre bundle $M^{2n}\rightarrow W=M^{2n}_{Id}\rightarrow \mathbb{S}^{1}$ , consider the map $l:M^{2n}\rightarrow W$, the 2-form define by $\omega=l^{*}\Omega$ is symplectic structure on $M^{2n}.$
\end{proof}
\subsection{\textbf{Gauge equation of torsionless selfdual connection(Levi-Civita connection) and existence of CoKhaler structure in three dimensional manifold.}}

\subsubsection{\textbf{Gauge equation selfdual torsionless connection}($\nabla^{lc}$)}

\begin{prop}
  The $2$-form $p_{\theta}$ is harmonic, i.e. $\Delta^{\text{lc}} p_{\theta} = 0$.
\end{prop}

\begin{proof}
For a torsionless connection $\nabla$ and a $k$-form $\omega$:
  \[
    d\omega\left(X_0,\dots, X_k\right) = 
    \sum_{i=0}^k (-1)^i \left(\nabla_{X_i}\omega\right)\left(X_0,\dots,\hat{X}_i,\dots,X_k\right)
    \]
Since $\nabla p_{\theta} = 0$, the previous formula applied to $p_{\theta}$ shows that $dp_{\theta} = 0$.
Let $\theta$ be a skew-symmetric solution gauge equation  $2$-form $p_{\theta}$ : $(X,Y) \mapsto p_{\theta}(X,Y) = g(\theta X,Y)$.
\[
\delta^{lc} p_{\theta}(Y_{1},....,Y_{r-1})=-\sum_{i=0}^{2n}(\nabla_{E_{i}}p_{\theta})(E_{i},Y_{1},....,Y_{r-1})
\]

\[\Delta^{\text{lc}} p_{\theta}=d(\delta^{lc} p_{\theta})+\delta^{lc}(dp_{\theta})=0.\]
\end{proof}

\subsubsection{\textbf{ Gauge equation solution and pseudo-Kahler structure. } }
The pseudo-Kahler manifold  were introduced by André Lichnerowicz in \cite{lichnerowicz1954varietes}.
\begin{defn}
An 2n-dimension manifold $(M,g,\Omega)$ is \textbf{pseudo-Kahler}, when we can define on it a Riemannian metric and a quadratic form $\Omega$ of rank 2n with zero covariant derivative in this metric.
 \end{defn}
 
\begin{prop}
Let $M$ be a $2n$ dimensional manifold. The following assertions are equivalent:
\begin{enumerate}
    \item $M$ admits a pseudo-Kahler structure,
    \item It exists a metric $g$ such that the gauge equation of selfdual torsionless connection on $M$ admits a skew-symmetric solution $\theta$ such that $\text{rank }\theta= 2n.$
\end{enumerate}
\end{prop}
\begin{proof}
Let us proves (1)$\Longrightarrow$(2): Assume that M admits a pseudo-Kahler structure $(\Omega,g)$, from the Definition 4.7, we have $\nabla^{lc}\Omega=0\quad\text{and}\quad \Omega^{n}\ne 0.$ There exist a skew-symmetric  $\theta$ of rank 2n, such that: $$\Omega(X,Y)=g(\theta X,Y)\quad \forall X,Y\in\mathcal{X}(M).$$ From the identity $$\nabla^{lc}\Omega=g(\nabla^{\text{lc}}_{Z}\theta X,Y)-g(\theta\nabla^{\text{lc}}_{Z}X,Y)$$ The condition $\nabla^{lc}\Omega=0$ implies that $\nabla^{lc}\theta=0.$\\
(2)$\Longrightarrow$(1): Let g be a metric on M and  by $\nabla^{lc}$ his levi-Civita connection. Let $\theta$ the skew-symmetric solution of the linear equation $\nabla^{lc}\theta=0$ such that the rank of $\theta$ is 2n. From proposition 4.3, we have $\nabla^{lc}p_{\theta}=0\quad\text{and}\quad p_{\theta}^{n}\ne 0.$ We deduce that $(g,p_{\theta})$ is pseudo-kahler structure on M.

\end{proof}

 \subsection{\textbf{Gauge equation solution and curvature} }
For a fixed $p \in M$, the Riemaniann metric $g$ admits an orthonormal basis $X_1,\dots,X_n$ in $T_pM$.With respect to it,
$\theta$ is represented by a skew-symmetric matrix $\Theta$ with entries $\Theta_{ij} =g(\theta X_{j},X_{i})$. 
It is well-known from elementary linear algebra that it exists a basis $Z_1,\dots Z_{2m},Z_{2m+1},\dots Z_{n}$ 
and real numbers $\lambda_1,\dots,\lambda_m$ such that:
\[
  \begin{split}
  & \Theta Z_{2k-1} = \lambda_k Z_{2k}, \, \Theta Z_{2k}= - \lambda_k Z_{2k-1}, \, k=1 \dots m \\
  & \Theta Z_{2m+k} = 0, \, k=1,\dots, n-2m
  \end{split}
  \]
Furthermore, the basis $Z_1,\dots,Z_n$ can be chosen to be orthonormal. This is due to the fact that in any case:
 $\Theta^2 Z_i = -\lambda_{k(i)}^2 Z_i$, where $\lambda$ is $0$ if $i>2m$ and $k(i) = \lfloor (i+1)/2 \rfloor$ otherwise.
 It thus comes:
 \[
   \met{\theta^{^2} Z_i}{Z_j} = -\lambda_k(i)^2 \met{Z_i}{Z_j}= \met{Z_i}{\theta^{^2}Z_j} = -\lambda_{k(j)} \met{Z_i}{Z_j} 
   \]
  if $\lambda_{k(i)} \neq \lambda_{k(j)} $,  then $\met{Z_i}{Z_j} = 0$. Otherwise, $Z_i,Z_j$ belong to the same linear subspace 
  of $T_pM$ and can thus be orthonormalized. In the $Z_i, i=1\dots n$ basis, the matrix $\Theta$ is block-diagonal, with $m$ blocks
  of the form:
  \[
    \begin{pmatrix}
      0 & -\lambda \\
      \lambda & 0
    \end{pmatrix}
    \]
    and the remaining entries all zero.
  \begin{rem}
    As a complex matrix, $\Theta$ is diagonal in the base 
    \[X_{2k-1}-iX_{2k},X_{2k-1}+iX_{2k}, k=0\dots m, Z_{2m+k}, \, k=1\dots m-2N\]
    with respective eigenvalues $i \lambda_k,, -i \lambda_k, 0$.
  \end{rem}
  \begin{prop}
    \label{prop:r_block_diag}
    For any $U,V \in T_pM$, the curvature tensor $R(U,V)$ is block diagonal in the basis $Z_i=1\dots n$.
  \end{prop}
\begin{proof}
  Let $U,V \in T_pM$ be fixed. In the basis $X_1,\dots,X_n$, $R(U,V)$ is represented by an skew symetric matrix, still denoted by 
   $R(U,V)$. Since $\nabla^{\text{lc}}_{X}\circ\theta=\theta\circ\nabla^{\text{lc}}_{X}$, $R(U,V)$ and $\Theta$ commute, and since they are both diagonalizable 
   (as complex matrices), they must have the same eigenspaces.
\end{proof}
\begin{rem}
  The proposition \ref{prop:r_block_diag} also shows that any gauge transformation $\theta^\prime$ satisfying $\nabla^{\text{lc}}_{X}\circ\theta^\prime=\theta^\prime\circ\nabla^{\text{lc}}_{X}$
  commutes with $R(U,V)$, and so is block diagonal in the base $Z_1,\dots,Z_n$. It must thus commute with $\theta$. 
\end{rem}
\begin{prop}
\label{prop:r_factorization}
The curvature tensor $R$ is such that:
\[
\begin{cases}
&R(Z_{2k},Z_{2k-1})Z_{2j}=-\mu_{kj} Z_{2j-1} \\
&R(Z_{2k-1},Z_{2k})Z_{2j}=\mu_{kj} Z_{2j-1} \\
&R(Z_{2k},Z_{2k-1})Z_{2j-1}=\mu_{kj} Z_{2j} \\
&R(Z_{2k-1},Z_{2k})Z_{2j-1}=-\mu_{kj} Z_{2j} \\
& 0 \text{ otherwise.} 
\end{cases}
\]
\end{prop}
\begin{proof}
Let us first recall that for any $X,Y,U,V$:
\[
\met{R(U,V)X}{Y}=\met{R(X,Y)U}{V}
\]
Then, using the expression of $R$ in the basis $Z_i, i=1\dots n$, it comes that only the terms:
\[
R(Z_{2k},Z_{2k-1}) = - R(Z_{2k-1},Z_{2k})
\]
can be non-zero. The claim follows by using the block diagonal expression of $R$.
\end{proof}
\begin{rem}
  \label{rem:ricci_gauge}
 A direct computation shows that the Ricci tensor is diagonal in the basis $Z_i, i=1\dots n$ and:
 \[
 \ricc(Z_{2k},Z_{2k})=\ricc(Z_{2k-1},Z_{2k-1}) = \mu_{kk}.
 \]
\end{rem}

\subsection{\textbf{ Gauge equation solution and K-cosymplectic Structures}
}
\begin{defn}
\cite{bazzoni2015k} A $2n+1$-dimensional manifold M is \textbf{K-cosymplectic} if it is endowed with a cosymplectic such that the Reeb vector field is Killing respect to some Riemannian metric on M.
\end{defn}
\begin{rem}
  By using Blair definition of cosymplectic manifold, Giovanni Bazzoni and Oliver Goertsches in \cite{bazzoni2015k} proves that the previous definition is equivalent to cosymplectic structure $(\theta,\xi,\eta,g)$ such that the Reeb vector field  $\xi$ is Killing.
\end{rem}
 \begin{prop}
 In a $2n+1$-dimensional oriented Riemannian manifold $(M,g)$,  if the gauge equation of selfdual torsionless(Levi-Civita connections) admits a skew-symmetric solution of rank $2n$. Then $M$ admits a K-cosymplectic structure.
 \end{prop}
  
 \begin{proof}
  Let $\theta$ be a skew-symmetric solution of the gauge equation. By assumption the rank of $\theta$ is 2n, so 
2-form $p_{\theta}$ has maximal rank, i.e. such that $p_{\theta}^{n}$  vanishes
nowhere.The gauge equation ($\nabla^{lc}\theta=0)$ implies that $\nabla^{lc}p_{\theta}=0(dp_{\theta}=0).$ The distribution $ker p_{\theta}$ is $\nabla^{lc}$-paralell, then associated to $p_{\theta}$ is its 1-dimensional kernel distribution(foliation) $ker p_{\theta}$. By using the orientation on $M$  together with $p_{\theta}^{n}$, we orient $ker p_{\theta}.$
  Let $\hat{\xi_{\theta}}$ be a unit norm section in $\ker p_\theta$. Denote by $\text{H}$ the mean curvature vector of the foliation $ker p_{\theta}$
 
  \[ \text{H}=(\nabla^{lc}_{ \hat{\xi_{\theta}}} \hat{\xi_{\theta}})\vert_{\ker p_\theta^{\perp}} 
  \]
 and $\eta_{\theta}$ be the volume form of $ker p_{\theta}$: \[\eta_{\theta}(X)=g(X,\hat{\xi_{\theta}})\quad \forall X\in \mathcal{X}(M)   \]
 We have by simple calculation 
 \[ 
 d\eta_{\theta}(\hat{\xi_{\theta}},X)=\hat{\xi}_{\theta}.<\hat{\xi_{\theta}},X>-X.|\hat{\xi_{\theta}}|^{2}-<\hat{\xi_{\theta}},[\hat{\xi_{\theta}},X]>
 \]
 \[  d\eta_{\theta}(\hat{\xi_{\theta}},X)=<{\nabla^{\text{lc}}}_{\hat{\xi_{\theta}}}\hat{\xi_{\theta}},X>-\frac{1}{2}X.|\hat{\xi_{\theta}}|^{2}=<\text{H},X>   \]
 The 1-dimensional foliation $\ker p_\theta$ is minimal foliation, then 
 
 \begin{equation}
     d\eta_{\theta}(\hat{\xi_{\theta}},X)=0\quad\forall X\in\mathcal{X}(M).
 \end{equation}
 The distribution $\text{ker}\eta_{\theta}$ is $\nabla^{lc}$-paralell, then $\text{ker}\eta_{\theta}$ is codimension one co-orientable foliation, by using the integrability condition: $$\eta_{\theta}([X,Y])=0\quad  \forall X,Y\in \Gamma(ker\eta_{\theta})$$ we deduce that 
 \begin{equation}
      d\eta_{\theta}(X,Y)=0 \quad  \forall X,Y\in \Gamma(ker\eta_{\theta})
 \end{equation}
 From (4.5) and (4.6)  we deduce that 
 \[ d\eta_{\theta}=0   \]
 Then $(p_{\theta},\eta_{\theta})$ is cosymplectic structure on M and $\hat{\xi_{\theta}}$ his Reeb vector field.
 
(i)  $${\nabla^{\text{lc}}}_{\hat{\xi_{\theta}}}\hat{\xi_{\theta}}=0.$$ The flows lines of $\hat{\xi_{\theta}}$ are geodesible flow.

(ii) By calculations 
\[ (L_{\hat{\xi_{\theta}}}g)(X,Y)=g(\nabla^{lc}_{X}\hat{\xi_{\theta}},Y)+g(X,\nabla^{lc}_{Y}\hat{\xi_{\theta}})=0\quad \forall X,Y\in\text{ker}\eta_{\theta}. \]
  Then $\hat{\xi_{\theta}}$ is Riemannian flow.\\ 
  From \cite{tondeur1997geometry}(proposition 10.10), (i) and (ii) implies that Reeb vector field $\hat{\xi_{\theta}}$ is Killing vector field ie($L_{\hat{\xi}_{\theta}}g=0).$
  
 \end{proof}
\begin{cor} Let M be a pseudo-Kahler manifold in sense of Lichnerowicz, the manifold
$W=M\times \mathbb{S}^{1}$ admits K-cosymplectic structures.
\end{cor}
\begin{proof}

  $W$ is a fiber bundle over $\mathbb{S}^{1}$, let   $\pi:W\rightarrow\mathbb{S}^{1}$ denote the natural projection on $\mathbb{S}^{1}$. Let $d\alpha$ be the angular form on $\mathbb{S}^{1}$ and $\frac{d}{d\alpha}$ its dual vector field on. It satisties  $d{\alpha}(\frac{d}{d\alpha})=1$ , and so induces  naturally on $W$ a non-vanishing closed 1-form  $\eta_{\alpha}=\pi^{\star}(d\alpha)$ and a non-vanishing vector field $\xi_{\alpha}$ such that: 
  \[   
  \eta_{\alpha}(\xi_{\alpha})= d{\alpha}(\frac{d}{d\alpha})=1.   
  \]
  By assumption  $M$ admit a pseudo-Kahler structure $(g,\Omega_{\theta})$, then on M we have : \[ \nabla\Omega_{\theta}=0\quad\text{and} \quad \Omega_{\theta}^{n}\ne 0 . \] 
  Let $p:W\rightarrow M$ denote the natural projection. Let denote by $\Bar{\Omega}_{\theta}$ the closed 2-form defined by: $$\Bar{\Omega}_{\theta}=p^{\star}\Omega_{\theta}.$$ We have \[  \Bar{\Omega}_{\theta}^{n}\wedge\eta_{\alpha}\ne0\quad \text{on}\quad W. \]
  
 $ker\Bar{\Omega}_{{\theta}_{p}}$ is one dimensional for all $p\in W$ and $\Bar{\Omega}_{\theta}$ determines a line bundle by: \[l_{\Bar{\Omega}_{\theta}}=\cup_{p\in W}(p,ker\Bar{\Omega}_{{\theta}_{p}})\]
   $ker(\eta_{\alpha})$ is a hyperplane distribution transverse to
   $l_{\Bar{\Omega}_{\theta}}$ and hence $\Bar{\Omega}_{\theta}$ restricts to a nondegenerate form on $ker(\eta_{\alpha})$.
   Let $\xi_{\alpha}$ to be the unique section of $l_{\Bar{\Omega}_{\theta}}$
   satisfying $\eta_{\alpha}(\xi_{\alpha}) = 1$.  We see that $\Bar{\Omega}_{\theta}^{n}\wedge\eta_{\alpha}\ne0$, so  the tangent bundle T$M$ splits 
   as the direct sum of a line bundle with a preferred nowhere vanishing section, and a symplectic vector bundle:
   \[  TM=\mathbb{R}\xi_{\theta}\oplus (ker\eta_{\alpha},\Bar{\Omega}_{\theta})  
   .\]
 Let  $h=g+(d\alpha)^{2}$ be a metric of W, $\xi_{\theta}$  is Killing for the metric $h$, then $(\Bar{\Omega}_{\theta},\eta_{\alpha})$ is K-cosymplectic structure on $W=M\times \mathbb{S}^{1}$.

\end{proof}

\subsubsection{\textbf{coKähler structure in dimension three and gauge equation solution}}

\begin{defn}
 An almost contact metric structure $(\theta,\xi,\eta,g)$  on an odd-dimensional smooth manifold $M$
is coKähler if it is cosymplectic and normal, that is $N_{\theta}+d\eta\otimes\xi=0$, where $N_{\theta}$ is the Nijenhuis torsion of $\theta$,  defined as:
$$ N_{\theta}(X,Y)=\theta^{2}[X,Y]-\theta([\theta X,Y]+[X,\theta Y])+[\theta X,\theta Y].$$
\end{defn}

As it is known, an almost contact metric structure is coKähler if and only if
both $\nabla^{lc}\eta=0$ and $\nabla^{lc}\Omega=0$, where $\nabla^{lc}$ is the covariant differentiation with respect $g$ and $\Omega$ the fundamental 2-form of the almost contact metric structure. From \cite{blair2010riemannian}(Theorem 6.7) we have the following assertion:
\begin{prop}
An cosymplectic  manifold $(M,\theta,\xi,\eta,g)$ is coKähler if and only if $\nabla^{lc}\theta=0.$
\end{prop}
  From \cite{li2008topology}, coKähler manifolds are odd-dimensional analog of Kähler manifolds:
  \begin{thm}\cite{li2008topology}
  Any coKähler manifold  is a Kähler mapping torus.
  \end{thm}
 coKähler manifolds coincide with cosymplectic
manifolds in Blair’s sense. 
\begin{thm}
Let $\text{M}$ be a 3-dimensional  manifold, the following assertions are equivalent:
   \begin{enumerate}
   \item  $\text{M}$ admits a coKähler structure(Kähler mapping torus),
    \item It exists a metric on $M$ such that gauge equation of the Levi-Civita connection admits a non-zero skew-symmetric solution .
   \end{enumerate}
\end{thm}
\begin{proof}

The necessary part (1) implies (2):
Assume that M admit a coKähler structure, then there exist a almost contact metric  structure $(\theta,\xi,\eta,g)$  on $\text{M}$  where $\eta$ is a 1-form, $\theta$  is an endomorphism of T M, $\xi$ is  a non-vanishing vector field
such that such that \[  \eta(\xi)=1  \quad\text{and}\quad \theta^{2}=-I+\eta\otimes\xi.\] The compatible Riemannian metric $g$  satifies :\[ g(\theta X,\theta Y)=g(X,Y)-\eta(X)\eta(Y) \quad\text{and}\quad  g(\theta X,Y)=-g(X,\theta Y),\] for any two vector fields $X, Y\in\mathcal{X}(\text{M})$. From \cite{blair2010riemannian}(Theorem 6.7) the Levi-Civita connection $\nabla^{lc}$ of the compabible metric $g$ satisfies $ \nabla^{lc}\theta=0.$ It exists a metric g on M such that gauge equation of the Levi-Civita connection admits a non-zero skew-symmetric solution.

Let proves the sufficient part (2) implies (1): Let $\theta$ be a skew-symmetric solution of the gauge equation($\nabla^{lc}\theta=0$). By assumption the rank of $\theta$ is 2, from the Proposition 3.7, we known that M admits a K-cosymplectic structures. From \cite{bazzoni2015k}[proposition 2.8] M admits a coKähler structure.

 \end{proof}

\begin{rem}
  Let  $(\text{M},g)$  be a closed orientable 3-dimensional  Riemannian manifold such that the gauge equation of Levi-Civita connection of g  admits a skew-symmetric solution. From work of Etienne Ghys \cite{ghys1983classification} \[   M\simeq \{\frac{\mathbb{T}^{2}\times[0,1]}{(x,0)\sim(A x,1)}\}\cup\{\textbf{Seifert fiber space}\}  \], where A is a Kahler isometry of $\mathbb{T}^{2}.$   
\end{rem}

\bibliographystyle{plain}
\bibliography{main.bib}
\end{document}